\documentclass[a4paper,12pt]{article}
\usepackage{times}
\usepackage[english]{babel}
\usepackage{amssymb,amsmath}
\numberwithin{equation}{section}
\usepackage{amsthm}
\usepackage{array}
\usepackage{wasysym}
\usepackage[noadjust]{cite}
\usepackage{hyperref}
\usepackage[height=22.7cm , width = 16cm , top = 4cm , left = 3cm, a4paper]{geometry}
\usepackage{amssymb}
\usepackage{amsmath}
\usepackage{cite}
 \usepackage{pstricks}
\usepackage{epsfig}
\usepackage{verbatim}
\usepackage{multicol}
\usepackage{graphicx, color, psfrag}
\usepackage{graphics, psfrag}
\usepackage{soul,color}
\usepackage[numbers,sort]{natbib}
 \newtheorem{theorem}{Theorem}[section]

\theoremstyle{definition}

\theoremstyle{ex}

\newcommand{\e}{\end{document}}

\begin{document}

\thispagestyle{empty}

\author{
{{\bf Beih S. El-Desouky, Abdelfattah Mustafa\footnote{Corresponding author: amelsayed@mans.edu.eg} {} and Shamsan AL-Garash}
\newline{\it{{}  }}
 } { }\vspace{.2cm}\\
 \small \it Department of Mathematics, Faculty of Science, Mansoura University, Mansoura 35516, Egypt.
}

\title{The Beta Flexible Weibull Distribution}

\date{}

\maketitle
\small \pagestyle{myheadings}
        \markboth{{\scriptsize The Beta Flexible Weibull Distribution}}
        {{\scriptsize {Beih S. El-Desouky, Abdelfattah Mustafa and Shamsan AL-Garash}}}

\hrule \vskip 8pt

\begin{abstract}
We introduce in this paper a new generalization of the flexible Weibull distribution with four parameters. This  model based on the Beta generalized (BG) distribution, Eugene et  al.   \cite{Eugeneetal2002}, they first using the BG distribution for generating new generalizations. This new model is called the beta flexible Weibull BFW distribution.  Some statistical properties such as the mode, the $r$th moment, skewness and kurtosis are derived. The moment generating function and the order statistics are obtained. Moreover, the estimations of the parameters are given by maximum likelihood method and the Fisher's information matrix is derived. Finally,  we study the advantage of the BFW distribution by an application using real data set.
\end{abstract}




\section{Introduction}
In recent years appeared many generalizations which based on the Weibull (WD) distribution, Weibull\cite{Weibull1951}, such as  Modified Weibull (MW) distribution  submitted by Lai et al. \cite{Laietal2003}, and Sarhan et al. \cite{SarhanandZaindin2009}, generalized modified Weibull (GMW) distribution,  Carrasco et al. \cite{Carrascoetal2008}.  Kumaraswamy Weibull (KW) distribution, Cordeiro et al. \cite{Cordeiroetal2010}, exponentiated modified Weibull extension (EMWE) distribution, Sarhan et al. \cite{SarhanandApaloo2013}. and flexible Weibull extension (FWE) distribution, introduced by Bebbington et  al. \cite{Bebbingtonetal2007}.\\

We said a random variable $X$ has a flexible Weibull (FW) distribution  with two parameter, if it's  cumulative distribution function (CDF) is given as follows
\begin{equation} \label{1}
 F_{FW}(x;\alpha ,\beta )=1-\exp\left\{-e^{\alpha x -\frac{\beta }{x}}\right\},  \quad  \alpha,\beta \; \text{and} \;\; x >0.
\end{equation}
Moreover the probability density function (PDF) corresponding Eq. (\ref{1}) is given by
\begin{equation} \label{2}
  f_{FW}(x;\alpha ,\beta )=(\alpha +\frac {\beta}{x^2})e^{\alpha x -\frac {\beta }{x}} \exp\left\{-e^{\alpha x -\frac{\beta }{x}}\right\},  \quad x>0.
\end{equation}
Some generalizations of the flexible Weibull distribution were discussed recently, such as  exponentiated flexible Weibull extension (EFWE) distribution, Weibull generalized flexible Weibull extension (WG-FWE) distribution, Mustafa et al.  \cite{Mustafaetal2016a}, kumaraswamy flexible Weibull extension (KFWE) distribution, El Damcese et al.  \cite{Damceseetal2016}, inverse flexible Weibull extension (IFWE) distribution and exponentiated generalized flexible Weibull extension (EG-FWE) distribution, Mustafa et al.  \cite{Mustafaetal2016b}.

\noindent
In this paper, we give a new generalization of the flexible Weibull distribution. This new generalization is called  beta flexible Weibull (BFW) distribution, using class of beta generalized (BG) distribution, Eugene et al. \cite{Eugeneetal2002}. By substituting $F(x,\varphi)$ given in Eq. (\ref{3}) to be $F_{FW}(x,\alpha , \beta )$ given in Eq. (\ref{1}).
\\
If $F(x,\varphi)$ the baseline CDF of a random variable, then the beta generalized distribution is given by
\begin{equation} \label{3}
F_{BG}(x;p,q,\varphi)=I_{F(x;\varphi)}(p,q)=\frac{1}{B(p,q)}\int\limits_{0}^{F(x,\varphi)}u^{p-1}(1-u)^{q-1}du,  \quad p, q>0,
\end{equation}
where $\varphi$ is the parameter vector, $I_y(p,q)=\frac{B_y(a,q)}{B(a,q)}$ is the incomplete beta function ratio, and
$$
B_y(p,q)=\int\limits_0^{y}u^{p-1}(1-u)^{q-1}du.
$$

\noindent
The PDF of the BG distribution corresponding Eq. (\ref{3}) is
\begin{equation} \label{4}
f_{BG}(x;p,q,\varphi)=\frac{f(x;\varphi)}{B(p,q)}F(x;\varphi )^{p-1}\left[1-F(x;\varphi )\right]^{q-1}.
\end{equation}

\noindent
The reliability function and failure rate function of the BG distribution corresponding Eq.(\ref{3}) are given, respectively, by the relations
\begin{equation} \label{5}
S_{BG}(x;p,q,\varphi)=1-F_{BG}(x;p,q,\varphi)=1-I_{F(x;\varphi )}(p,q),
\end{equation}

\noindent
and
\begin{equation} \label{6}
h_{BG}(x;p,q,\varphi)=\frac{f(x;\varphi)F(x;\varphi )^{p-1}\left[1-F(x;\varphi )\right]^{q-1}}{B(p,q)\left[ 1-I_{F( x;\varphi )}(p,q)\right ]}.
\end{equation}

\noindent
Using expression $B(p,q)=\frac{\Gamma (p).\Gamma (q)}{\Gamma (p+q)}$, and the series representation
$$
(1-u)^{q-1}=\sum_{i=0}^{q-1} \frac{(-1)^i \Gamma (q)}{i! \Gamma (q-i)} u^i,\quad \text{for real non-integer} \quad |u|<1, \quad q>0,
$$

\noindent
the  CDF of the BG distribution in Eq. (\ref{3}) can be rewritten as follows
\begin{equation} \label{7}
F_{BG}(x;p,q,\varphi)=\frac{\Gamma (p+q)}{\Gamma (p)}\sum_{i=0}^{q-1 }\frac{(-1)^i\left[F(x;\varphi )\right]^{p+i}}{i! (p+i) \Gamma (q-i)},  \quad p, q>0.
\end{equation}

Recently the class of BG  distribution Eq. (\ref{3}) has been receiving considerable attention. This distribution was firstly studied by Eugene et al. \cite{Eugeneetal2002}. Many researchers considered different forms of distribution function $F(x,\varphi)$ given in Eq. (\ref{3}) and studied their properties.  Beta normal (BN) distribution, Eugene et al. \cite{Eugeneetal2002}, beta gumbel (BG) distribution  has been introduced by Nadarajah and Kotz \cite{Nadarajahkotz2004},  beta Weibull (BW) distribution,  Famoye et al. \cite{Famoyeetal2005}, beta modified Weibull (BMW) distribution,  Silva et al. \cite{Silvaetal2010} and  beta exponential (BE) distribution, Nadarajah et al. \cite{Nadarajahkotz2006}.\\

This paper is arranged as follows, the distribution function, density function and failure rate function of the BFW distribution  are defined in Section 2. Some statistical properties including, the mode, $rth$ moment, skewness and kurtosis are presented in Sections 3. The moment generating function is derived in Section 4. The order statistics is discussed in Section 5. The maximum likelihood estimation MLEs of the parameters is obtained in Section 6. Real data set are analyzed in Section 7. Moreover,  we discuss the results and compare it with existing distributions. Finally,  we give a conclusion  in Section 8.


\section{The Beta Flexible Weibull Distribution}
We define the BFW($\alpha, \beta, p, q $) distribution, by substituting  $F(x,\varphi)$ in Eq. (\ref{7}) to be the distribution function $F_{FW}(x; \alpha, \beta )$  of the FWD in Eq. (\ref{1}). So  the CDF of the Beta Flexible Weibull distribution is
\begin{equation} \label{14}
F_{BFW}(x;\eta  ) = \frac{\Gamma (p+q)}{\Gamma (p)}\sum_{i=0}^{q-1}\frac{(-1)^i \left[1 - e^{-e^{\alpha x-\frac{\beta }{x}}} \right ]^{p+i}}{i! \Gamma (q-i) (p+i)}, \;  \forall x>0,
\end{equation}

\noindent
where $\eta = (\alpha ,\beta, p, q )$, $\alpha, \beta, p>0$ and $q>0$.
\\
Replacing $F_{FW}(x;\alpha ,\beta )$ in Eq. (\ref{1}) and $f_{FW}(x;\alpha ,\beta)$ in Eq. (\ref{2}) by the $F(x,\varphi)$ and $f(x;\varphi)$ in Eq. (\ref{4}), respectively, the PDF corresponding Eq.(\ref{14}) can be obtained as follows
\begin{equation} \label{15}
f_{BFW}(x;\eta  ) = \frac{\Gamma (p+q)}{\Gamma (p).\Gamma (q)} \left(\alpha +\frac{\beta }{x^2}\right) e^{\alpha x -\frac{\beta }{x}} e^{-q e^{\alpha x -\frac{\beta }{x}}} \left[1 - e^{-e^{\alpha x-\frac{\beta }{x}}} \right ]^{p-1}.
\end{equation}

The reliability function $S_{BFW}(x)$ and failure rate $h_{BFW}(x)$ function of $X$ having  the BFWD$(\eta )$, respectively, are given by
\begin{equation} \label{16}
S_{BFW}(x;\eta ) =1- \frac{\Gamma (p+q)}{\Gamma (p)}\sum_{i=0}^{q-1}\frac{(-1)^i \left[1 - e^{-e^{\alpha x-\frac{\beta }{x}}} \right ]^{p+i}}{i! \Gamma (q-i) (p+i)}, \; \forall x>0,
\end{equation}
and
\begin{equation} \label{17}
h_{BFW}(x;\eta )= \frac{\frac{\Gamma (p+q)}{\Gamma (p).\Gamma (q)} \left(\alpha +\frac{\beta }{x^2}\right) e^{\alpha x -\frac{\beta }{x}} e^{-q e^{\alpha x -\frac{\beta }{x}}} \left[1 - e^{-e^{\alpha x-\frac{\beta }{x}}} \right ]^{p-1}}{1- \frac{\Gamma (p+q)}{\Gamma (p)} \sum_{i=0}^{q-1}\frac{(-1)^i \left[1 - e^{-e^{\alpha x-\frac{\beta }{x}}} \right ]^{p+i}}{i! \Gamma (q-i) (p+i)}}.
\end{equation}

\noindent
Also, the reversed failure rate $r_{BFW}(x)$ and cumulative-failure rate $H_{BFW}(x)$ functions of $X$ having  the BFWD$(\eta )$, respectively, are given by
\begin{equation} \label{18}
r_{BFW}(x;\eta  ) = \frac{\frac{\Gamma (p+q)}{\Gamma (p).\Gamma (q)} \left(\alpha +\frac{\beta }{x^2}\right) e^{\alpha x -\frac{\beta }{x}} e^{-q e^{\alpha x -\frac{\beta }{x}}} \left[1 - e^{-e^{\alpha x-\frac{\beta }{x}}} \right ]^{p-1}}{\frac{\Gamma (p+q)}{\Gamma (p)}\sum_{i=0}^{q-1}\frac{(-1)^i \left[1 - e^{-e^{\alpha x-\frac{\beta }{x}}} \right ]^{p+i}}{i! \Gamma (q-i) (p+i)}},
\end{equation}
and
\begin{equation} \label{19}
H_{BFW}(x;\eta )= \int_0^x h_{BFW}(t)dt= \int_0^x \frac{\frac{\Gamma (p+q)}{\Gamma (p).\Gamma (q)} \left(\alpha +\frac{\beta }{t^2}\right) e^{\alpha t -\frac{\beta }{t}} e^{-q e^{\alpha t -\frac{\beta }{t}}} \left[1 - e^{-e^{\alpha t-\frac{\beta }{t}}} \right ]^{p-1}}{1- \frac{\Gamma (p+q)}{\Gamma (p)}\sum_{i=0}^{q-1}\frac{(-1)^i \left[1 - e^{-e^{\alpha t-\frac{\beta }{t}}} \right ]^{p+i}}{i! \Gamma (q-i) (p+i)}}dt.
\end{equation}

Figures \ref{fig:1}, \ref{fig:2} and \ref{fig:3} display functions the CDF, PDF, $S_{BFW}(x)$, $h_{BFW}(x)$, $r_{BFW}(x)$ and $H_{BFW}(x)$ of the BFWD$(\eta )$ for different values of parameters.

\begin{figure} [!h]
\begin{center}
\includegraphics[scale = 0.55]{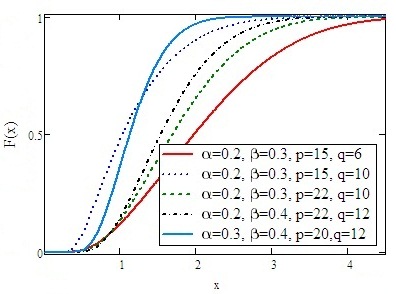}
\includegraphics[scale = 0.55]{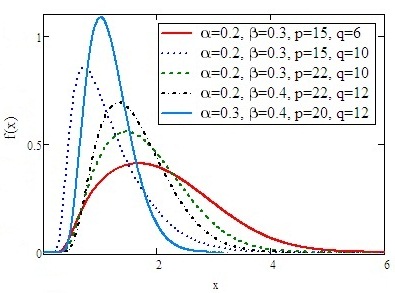}
\end{center}
\caption{Shows the CDF and PDF of the BFWD for different values of parameters}
\label{fig:1}
\end{figure}

\begin{figure} [!h]
\begin{center}
\includegraphics[scale = 0.55]{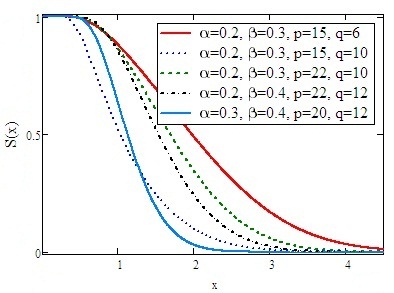}
\includegraphics[scale = 0.55]{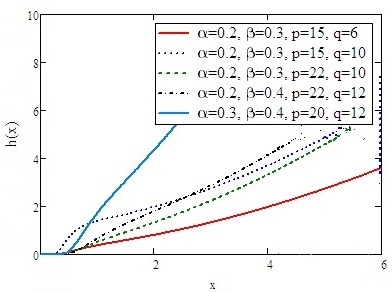}
\end{center}
\caption{Shows the $S_{BFW}(x)$ and $h_{BFW}(x)$ of the BFWD for different values of parameters}
\label{fig:2}
\end{figure}

\begin{figure} [!h]
\begin{center}
\includegraphics[scale = 0.55]{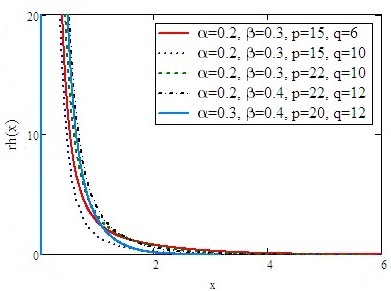}
\includegraphics[scale = 0.55]{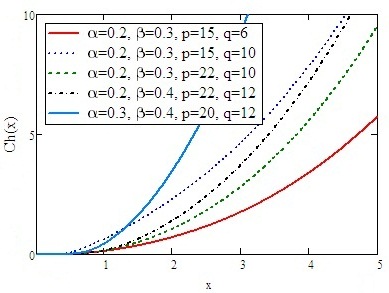}\\
\end{center}
\caption{Shows the $r_{BFW}(x)$ and $H_{BFW}(x)$ of the BFWD for different values of parameters}
\label{fig:3}
\end{figure}
\section{Statistical Properties}
Some statistical properties for the BFWD$(\eta) $, such as the mode, the $r$th moment,  skewness and kurtosis are given as follows.

\subsection{The Mode of the BFWD}
The mode of the BFWD$(\eta )$  can be obtained by differentiating its probability density function with respect to  $x$ in Eq. (\ref{15})  and equating it to zero.
\begin{equation*}
f^{'}(x; \eta )=0.
\end{equation*}
So the mode of the BFWD$(\eta) $ is  solution of the following relation
\begin{eqnarray} \label{mod9}
&&
\frac{\Gamma (p+q)}{\Gamma (p).\Gamma (q)} e^{\alpha x -\frac{\beta }{x}} e^{-q e^{\alpha x -\frac{\beta }{x}}} \left[ 1 - e^{-e^{\alpha x-\frac{\beta }{x}}} \right ]^{p-1}\times \nonumber \\ & & \left\{ \frac{-2\beta }{x^3}+(\alpha +\frac{\beta }{x^2})^2\left[1-e^{\alpha x -\frac{\beta }{x}}\left ( q+\frac{(p-1)}{e^{e^{\alpha x-\frac{\beta }{x}}}-1}\right )\right ] \right\}=0.
\end{eqnarray}%

\noindent
The BFWD$(\eta) $ has exactly only one peak, so this generalization is a unimodal. Figure (2) shows that  Eq. (\ref{mod9}) has only one solution. We can solve it numerically.

\subsection{The Moments}
The $rth$ moment for BFWD$(\eta )$ is given by  Theorem (\ref{Th1}). The gamma function with negative integer is defined by
$$
\Gamma (-r)=\frac{(-1)^r}{r!}\left[ \phi (r)-\gamma \right ],
$$
where $\gamma $ denote Euler's constant which is defined as
$$
\gamma =\lim_{n\rightarrow \infty } \sum_{k=1}^{n} \left ( \frac{1}{k}-\ln(n)\right ), \qquad \phi (r)=\sum_{i=1}^{r}\frac{1}{i},
$$
see  Fisher and Kilicman \cite{Fisheretal2012}.
\begin{theorem} \label{Th1}
If $X$ random variable having the BFWD$(\eta )$, so the $r$th moment of  $X$, is given by
\begin{eqnarray}  \label{26}
\mu_ r^{'} & = &  \frac{\Gamma (p+q)}{\Gamma (q)}\sum_{n=0}^{\infty}\sum_{m=0}^{\infty}\sum_{\ell=0}^{\infty} \frac{(-1)^{n+m}(q+n)^m (m+1)^{r+2\ell +1} \alpha ^\ell \beta ^{r+\ell+1}}{n! m! \ell!\Gamma (p-n)}\times
\nonumber\\
&& \hspace{3 cm}
\left[ \alpha \Gamma(-r-\ell-1)+\frac{\Gamma(-r-\ell+1)}{(m+1)^2 \beta}\right ].
\end{eqnarray}
\end{theorem}

\begin{proof}
The $r$th moment of the random variable $X$ with the PDF is given by
\begin{equation} \label{27}
\mu_ r^{'}=\int\limits_{0}^{\infty} x^r f(x) dx.
\end{equation}
	
Using PDF for the BFWD$(\eta )$ from Eq. (\ref{15}) into Eq. (\ref{27}) we get
\begin{equation*}
\mu_ r^{'}= \int_0^\infty x^r \frac{\Gamma (p+q)}{\Gamma (p).\Gamma (q)} \left(\alpha +\frac{\beta }{x^2}\right) e^{\alpha x -\frac{\beta }{x}} e^{-q e^{\alpha x -\frac{\beta }{x}}} \left[1 - e^{-e^{\alpha x-\frac{\beta }{x}}} \right ]^{p-1} dx,
\end{equation*}
	
\noindent
where  the
$\left[1 - e^{-e^{\alpha x-\frac{\beta }{x}}} \right ]^{p-1}$  can be written as
$$
\left[1 - e^{-e^{\alpha x-\frac{\beta }{x}}} \right ]^{p-1}=\sum_{n=0}^{\infty}\frac{(-1)^n \Gamma (p)}{n! \Gamma (p-n)} e^{-n e^{\alpha x -\frac{\beta }{x}}},
$$
then we get
\begin{equation*}
\mu_ r^{'}= \frac{\Gamma (p+q)}{\Gamma (p).\Gamma (q)} \sum_{n=0}^{\infty}\frac{(-1)^n \Gamma (p)}{n! \Gamma (p-n)} \int_0^{\infty} x^r \left[\alpha  +\frac{\beta }{x^2}\right] e^{\alpha x - \frac{\beta }{x}} e^{-(q+n)e^{\alpha x - \frac{\beta }{x}}} dx,
\end{equation*}
\noindent
the function $e^{-(q+n) e^{\alpha x - \frac{\beta }{x}}}$, can be written as
\[
e^{-(q+n) e^{\alpha x - \frac{\beta }{x}}} =\sum_{m=0}^{\infty}\frac{(-1)^m(q+n)^m}{m!} e^{m(\alpha x -\frac{\beta }{x})},
\]
we obtain
\begin{equation*}
\mu_ r^{'} =  \frac{\Gamma (p+q)}{\Gamma (p).\Gamma (q)}\sum_{n=0}^{\infty}\sum_{m=0}^{\infty} \frac{(-1)^{n+m} \Gamma (p)(q+n)^m}{n! m! \Gamma (p-n)} \int_0^{\infty} x^r \left[\alpha  +\beta x^{-2}\right] e^{(m+1)\alpha  x} e^{-(m+1)\frac{\beta }{x}} dx,
\end{equation*}
	
\noindent
using series expansion of  $e^{(m+1)\alpha  x}$,
\[
e^{(m+1)\alpha x}= \sum_{\ell=0}^{\infty}\frac{(m+1)^\ell \alpha  ^\ell x^\ell }{\ell!},
\]
we have
\begin{eqnarray*}
\mu_ r^{'} & = &  \frac{\Gamma (p+q)}{\Gamma (p).\Gamma (q)}\sum_{n=0}^{\infty}\sum_{m=0}^{\infty}\sum_{\ell=0}^{\infty} \frac{(-1)^{n+m} \Gamma (p)(q+n)^m (m+1)^\ell \alpha  ^\ell}{n! m! \ell!\Gamma (p-n)}\times
\nonumber\\
&& \hspace{3 cm}
\left [ \int_0^{\infty} \alpha x^{r+\ell} e^{-(m+1)\frac{\beta }{x}}  dx
+ \int_0^{\infty}\beta x^{r+\ell-2} e^{-(m+1)\frac{\beta }{x}}  dx\right ],
\end{eqnarray*}
	
where gamma function is defined by ( see Zwillinger \cite{Zwillinger2014}),
\[
\Gamma (\nu  )=x^\nu   \int_0^{\infty}e^{-tx} t^{\nu  -1}dt, \quad x>0.
\]

Finally the $r$th moment of the BFWD$(\eta )$ can be obtained as follows
\begin{eqnarray*}
\mu_ r^{'} & = &  \frac{\Gamma (p+q)}{\Gamma (q)}\sum_{n=0}^{\infty}\sum_{m=0}^{\infty}\sum_{\ell=0}^{\infty} \frac{(-1)^{n+m}(q+n)^m (m+1)^{r+2\ell +1} \alpha ^\ell \beta ^{r+\ell+1}}{n! m! \ell!\Gamma (p-n)}\times
\nonumber\\
&& \hspace{3 cm}
\left[ \alpha \Gamma(-r-\ell-1)+\frac{\Gamma(-r-\ell+1)}{(m+1)^2 \beta}\right ].
\end{eqnarray*}
This completes the proof.
	
\end{proof}

Furthermore, using the first four moments in Eq. (\ref{26}), the measures of skewness and kurtosis of the BFWD$(\eta )$ can be expressed in the following relations, (see Bowley \cite{Bowley1920}).
\\
\noindent
The skewness of the BFWD$(\eta )$
\begin{equation}
Sk(X)=\frac{E(X^3)- 3\mu  E(X^2) +2\mu^3}{\sigma^3}.
\end{equation}


\noindent
The kurtosis of the BFWD$(\eta )$
\begin{equation}
Ku(X)=\frac{E(X^4)-4\mu  E(X^2) + 6\mu^2  E(X^2) - 3\mu^4}{\sigma^4}.
\end{equation}

%


\section{The Moment Generating Function}
We derive the moment generating function MGF of the BFWD$(\eta )$  by Theorem (\ref{Th2}) .

\begin{theorem} \label{Th2}
If $X$ is a random variable having  the BFWD$(\eta )$, so the  moment generating function $MGF$ of $X$ is given by
\begin{eqnarray}
M_X(t) & = &  \frac{\Gamma (p+q)}{\Gamma (q)}\sum_{n=0}^{\infty}\sum_{m=0}^{\infty}\sum_{\ell=0}^{\infty}\sum_{r=0}^{\infty} \frac{(-1)^{n+m}(q+n)^m (m+1)^{r+2\ell +1} \alpha ^\ell \beta ^{r+\ell+1}t^r}{n! m! \ell! r!\Gamma (p-n)}\times
\nonumber\\
&& \hspace{3 cm}
\left[ \alpha \Gamma(-r-\ell-1)+\frac{\Gamma(-r-\ell+1)}{(m+1)^2 \beta}\right ].
\end{eqnarray}
\end{theorem}
\begin{proof}
The moment generating function $MGF$ of the random variable $X$ with the PDF is
\begin{equation} \label{28}
M_X(t)=\int\limits_{0}^{\infty} e^{tx} f(x) dx.
\end{equation}

\noindent
Since   $$e^{tx}= \sum_{r=0}^{\infty} \frac{t^r x^r}{r!},$$ then we get
\begin{equation} \label{28a}
M_X(t) = \sum_{r=0}^\infty \frac{t^r}{r!} \int_0^{\infty} x^r f(x) dx.
=
\sum_{r=0}^\infty \frac{t^r}{r!} \mu_r^{'}.
\end{equation}

\noindent
Using $\mu_r^{'}$ from Eq. (\ref{26}) into Eq. (\ref{28a}) the $MGF$  of the BFWD$(\eta )$ can be obtained as follows
\begin{eqnarray*}
M_X(t) & = &  \frac{\Gamma (p+q)}{\Gamma (q)}\sum_{n=0}^{\infty}\sum_{m=0}^{\infty}\sum_{\ell=0}^{\infty}\sum_{r=0}^{\infty} \frac{(-1)^{n+m}(q+n)^m (m+1)^{r+2\ell +1} \alpha ^\ell \beta ^{r+\ell+1}t^r}{n! m! \ell! r!\Gamma (p-n)}\times
\nonumber\\
&& \hspace{3 cm}
\left[ \alpha \Gamma(-r-\ell-1)+\frac{\Gamma(-r-\ell+1)}{(m+1)^2 \beta}\right ].
\end{eqnarray*}
		
\noindent
This completes the proof.
\end{proof}

\section{The Order Statistics }
Suppose $X_{1:n},X_{2:n},\cdots,X_{n:n}$ denote the order statistics obtained from a random sample $X_1$, $X_2$, $\cdots$, $X_n$ taken from a continuous population with distribution function $F_{BFW}(x;\eta )$ and density function  $f_{BFW}(x;\eta )$, so the PDF of $X_{r:n}$ can be written as
\begin{equation} \label{29}
f_{r:n}(x;\eta )=\frac{1}{B(r,n-r+1)}\left[F_{BFW}(x;\eta )\right]^{r-1} \left[1-F_{BFW}(x;\eta)\right]^{n-r} f_{BFW}(x;\eta ),
\end{equation}
	
\noindent
where $F_{BFW}(x;\eta )$ and  $f_{BFW}(x;\eta )$ are the CDF and PDF of the BFWD$(\eta )$ given by Eq. (\ref{14}) and Eq. (\ref{15}), respectively.  Can be defined first order statistics as $X_{1:n}= \min(X_1,X_2,\cdots,X_n)$, and the last order statistics as $X_{n:n}= \max (X_1,X_2,\cdots,X_n)$, where   $0<F_{BFW}(x;\eta)<1,$  $\forall x>0.$
\\
The binomial expansion of $[1-F_{BFW}(x;\eta)]^{n-r}$ can be expressed as
\begin{equation} \label{30}
\left[1-F_{BFW}(x;\eta)\right]^{n-r}=\sum_{i=0}^{n-r}\begin{pmatrix} n-r \\ i \end{pmatrix} (-1)^i [F_{BFW}(x;\eta )]^i.
\end{equation}

\noindent
Using Eq. (\ref{30}) in Eq. (\ref{29}), we get
\begin{equation} \label{31}
f_{r:n}(x;\eta )=\frac{1}{B(r,n-r+1)} f_{BFW}(x;\eta )\sum_{i=0}^{n-r}\begin{pmatrix} n-r \\ i \end{pmatrix} (-1)^i \left[F_{BFW}(x;\eta)\right]^{i+r-1}.
\end{equation}

\noindent
Using Eq. (\ref{14}) and Eq. (\ref{15}) in Eq. (\ref{31}), we obtain
\begin{equation} \label{32}
f_{r:n}(x;\eta )=\sum_{i=0}^{n-r} \frac{(-1)^i n!}{i! (r-1)! (n-r-i)!} \left[F_{BFW}(x,\eta )\right]^{i+r-1} f_{BFW}(x;\eta ).
\end{equation}

\noindent
Equation (\ref{32}) shows that the $f_{r:n}(x;\eta )$ is the weighted average of the BFWD$(\eta )$ with various shape parameters.

\section{Parameters Estimation}
Parameters estimation of the BFWD$(\alpha ,\beta , p, q)$, point and interval estimation  are derived by using maximum likelihood estimation MLE method with a complete sample.

\subsection{Maximum likelihood estimation}
Let $x_1,x_2,\cdots ,x_n$ denote a random sample of complete data from the BFWD$(\eta )$. The Likelihood function $L$ is defined as follows
\begin{equation} \label{33}
L = \prod_{i=1}^{n} f(x_i;\vartheta).
\end{equation}
	
Using Eq.(\ref{15}) in Eq.(\ref{33}), we have
\begin{equation*}
L = \prod_{i=1}^{n}\frac{\Gamma (p+q)}{\Gamma (p).\Gamma (q)} \left(\alpha +\frac{\beta }{x_i^2}\right) e^{\alpha x_i -\frac{\beta }{x_i}} e^{-q e^{\alpha x_i -\frac{\beta }{x_i}}} \left[1 - e^{-e^{\alpha x_i-\frac{\beta }{x_i}}} \right ]^{p-1}.
\end{equation*}
	
The log-likelihood function is
\begin{equation} \label{34}
\mathcal{L} = n \left [ \ln\left(\frac{\Gamma (p+q)}{\Gamma (p).\Gamma (q)}\right)\right ] + \sum_{i=1}^{n} \ln \left ( \alpha +\frac{\beta }{x_i^2} \right ) +  \sum_{i=1}^{n} \left(\alpha x_i -\frac{\beta }{x_i}\right) -q \sum_{i=1}^{n} e^{\alpha x_i -\frac{\beta }{x_i}} +(p-1) \sum_{i=1}^{n} \ln\left(1-e^{-e^{\alpha x_i-\frac{\beta }{x_i}}}\right).
\end{equation}

The MLE of the parameters $\alpha, \beta, p $ and $q$ are obtained by differentiating the  function $\mathcal{L}$ in Equation (\ref{34}) with respect to  $\alpha, \beta, p $ and $q$, then equating it to zero, as follows
\begin{eqnarray} \label{35}
\frac{\partial \mathcal{L}}{\partial \alpha } & = &
\sum_{i=1}^{n}\frac{ x_i^2}{\beta+\alpha x_i^2 }+\sum_{i=1}^{n} x_i - q\sum_{i=1}^{n} x_i e^{\alpha x_i-\frac{\beta }{x_i}} + (p-1)  \sum_{i=1}^{n} \frac{x_i e^{\alpha x_i-\frac{\beta }{x_i}}}{e^{e^{\alpha x_i-\frac{\beta }{x_i}}} -1} = 0,
\\ \label{36}
\frac{\partial \mathcal{L}}{\partial \beta  } & = &
\sum_{i=1}^{n}\frac{1}{\beta+\alpha x_i^2}-\sum_{i=1}^{n} \frac{1}{x_i} + q \sum_{i=1}^{n} \frac{1}{x_i} e^{\alpha x_i-\frac{\beta }{x_i}} - (p-1)  \sum_{i=1}^{n} \frac{e^{\alpha x_i-\frac{\beta }{x_i}}}{x_i\left( e^{e^{\alpha x_i - \frac{\beta }{x_i}}}-1\right)} = 0,
\\ \label{37}
\frac{\partial \mathcal{L} }{\partial p } &=&
n\psi _p (p+q) - n \psi (p) + \sum_{i=1}^{n} \ln \left(1- e^{-e^{\alpha x_i - \frac{\beta }{x_i}}}\right) = 0,
\\ \label{38}
\frac{\partial \mathcal{L} }{\partial q } &=&
n\psi _q (p+q) - n \psi (q) - \sum_{i=1}^{n} e^{\alpha x_i-\frac{\beta }{x_i}} = 0.
\end{eqnarray}

The MLEs can be obtained by solving the equations, (\ref{35})-(\ref{38}), numerically for $\alpha, \beta, p$ and $q$ by using Mathcad or Maple software.

\subsection{Asymptotic confidence bounds}
The asymptotic confidence intervals can be obtained by using variance covariance matrix $I^{-1}$, since the parameters ($\alpha, \beta, p, q$)  are positive and the $I^{-1}$ is inverse of the observed information matrix given by

\begin{eqnarray} \label{39}
\mathbf{I^{-1}} &= &
\left(
\begin{array}{cccc}
-\frac{\partial ^2 \mathcal{L}}{\partial \alpha  ^2} & -\frac{\partial ^2 \mathcal{L}}{\partial \alpha  \partial \beta } & -\frac{\partial ^2 \mathcal{L}}{\partial \alpha \partial p } &-\frac{\partial ^2 \mathcal{L}}{\partial \alpha \partial q }
\\
-\frac{\partial ^2 \mathcal{L}}{\partial \beta  \partial \alpha } & -\frac{\partial ^2 \mathcal{L}}{\partial \beta^2} & -\frac{\partial ^2 \mathcal{L}}{\partial \beta \partial p}  & -\frac{\partial ^2 \mathcal{L}}{\partial \beta \partial q}
\\
-\frac{\partial ^2 \mathcal{L}}{\partial p  \partial \alpha } & -\frac{\partial ^2 \mathcal{L}}{\partial p  \partial \beta } & -\frac{\partial ^2 \mathcal{L}}{\partial p ^2} & -\frac{\partial ^2 \mathcal{L}}{\partial p  \partial q }
\\
-\frac{\partial ^2 \mathcal{L}}{\partial q  \partial \alpha } & -\frac{\partial ^2 \mathcal{L}}{\partial q  \partial \beta } & -\frac{\partial ^2 \mathcal{L}}{\partial q  \partial p } & -\frac{\partial ^2 \mathcal{L}}{\partial q^2}
\end{array}
\right)^{-1}
=
\left(
\begin{array}{cccc}
var(\hat{\alpha }) & cov( \hat{\alpha }, \hat{\beta }) & cov( \hat{\alpha }, \hat{ p})& cov( \hat{\alpha }, \hat{ q})
\\
cov( \hat{\beta },\hat{\alpha  }) & var( \hat{\beta }) & cov( \hat{\beta }, \hat{ p })& cov( \hat{\beta }, \hat{ q  })
\\
cov( \hat{ p}, \hat{\alpha }) & cov( \hat{ p}, \hat{\beta }) &  var( \hat{ p})& cov( \hat{ p}, \hat{q })
\\
cov( \hat{ q}, \hat{\alpha }) & cov( \hat{ q}, \hat{\beta }) & cov( \hat{ q}, \hat{p})&  var( \hat{ q})
\end{array}
\right),
\nonumber\\
\end{eqnarray}

\noindent
where
\begin{eqnarray} \label{40}
\frac{\partial ^2 \mathcal{L}}{\partial \alpha^2 } &= &
-\sum_{i=1}^{n}\frac{x_i^4}{\left(\beta+\alpha x_i^2\right)^2} - q\sum_{i=1}^{n} x_i^2 e^{\alpha x_i -\frac{\beta }{x_i}} + (p-1)\sum_{i=1}^n x_i H_i,
\\ \label{41}
\frac{\partial ^2 \mathcal{L}}{\partial \alpha \partial \beta } &= &
-\sum_{i=1}^{n}\frac{x_i^2}{\left(\beta+\alpha x_i^2\right)^2} + q\sum_{i=1}^{n} e^{\alpha x_i -\frac{\beta }{x_i}} - (p-1) \sum_{i=1}^n H_i,
\\ \label{42}
\frac{\partial ^2 \mathcal{L}}{\partial \alpha  \partial p} &= &
\sum_{i=1}^{n} \frac {x_i e^{\alpha x_i-\frac{\beta }{x_i}}}{\left[ e^{e^{\alpha x_i-\frac{\beta }{x_i}}}-1\right]},
\\\label{43}
\frac{\partial ^2 \mathcal{L}}{\partial \alpha  \partial q} &= &
-\sum_{i=1}^{n} x_i e^{\alpha x_i-\frac{\beta }{x_i}},
\end{eqnarray}
\begin{eqnarray} \label{44}
\frac{\partial ^2 \mathcal{L}}{\partial \beta^2 } &= &
-\sum_{i=1}^{n}\frac{1}{\left(\beta+\alpha x_i^2\right)^2} -q \sum_{i=1}^{n}\frac{e^{\alpha x_i -\frac{\beta }{x_i}}}{x_i^2} + (p-1)\sum_{i=1}^n \frac{H_i}{x_i^2},
\\ \label{45}
\frac{\partial ^2 \mathcal{L}}{\partial \beta  \partial p  }&= &
-\sum_{i=1}^{n}\frac{e^{\alpha x_i-\frac{\beta }{x_i}}}{x_i \left[ e^{e^{\alpha x_i-\frac{\beta }{x_i}}} -1\right]},
\\ \label{46}
\frac{\partial ^2 \mathcal{L}}{\partial \beta  \partial q  }&= &
\sum_{i=1}^{n}\frac{e^{\alpha x_i-\frac{\beta }{x_i}}}{x_i},
\\ \label{47}
\frac{\partial ^2 \mathcal{L}}{\partial p^2} &=&
n\frac{\partial ^2}{\partial p^2} \ln \left[ \Gamma (p+q)\right ] - n \frac{\partial ^2}{\partial p^2} \ln \left [ \Gamma (p)\right ],
\\ \label{48}
\frac{\partial ^2 \mathcal{L}}{\partial p  \partial q  }&= &
n\frac{\partial ^2}{\partial p \partial q} \ln \left [ \Gamma (p+q)\right ],
\\ \label{49}
\frac{\partial ^2 \mathcal{L}}{\partial q^2} &=&
n\frac{\partial ^2}{\partial q^2} \ln \left [ \Gamma (p+q)\right ] - n \frac{\partial ^2}{\partial q^2} \ln \left[ \Gamma (q)\right ],
\end{eqnarray}

where $$H_i= \frac{x_i e^{\alpha x_i-\frac{\beta }{x_i}}\left[e^{e^{\alpha x_i-\frac{\beta }{x_i}}}\left(1-e^{\alpha x_i-\frac{\beta }{x_i}}\right)-1\right ]}{\left[e^{e^{\alpha x_i-\frac{\beta }{x_i}}} -1\right]^2}.$$

Moreover, the $(1-\lambda )100\%$ confidence intervals of the four parameters can be obtained by using variance matrix as follows
$$
\hat{\alpha} \pm Z_{\frac{\lambda }{2}}\sqrt{var(\hat{\alpha })},\quad \hat{\beta} \pm Z_{\frac{\lambda}{2}}\sqrt{var(\hat{\beta})}, \quad \hat{p} \pm Z_{\frac{\lambda}{2}}\sqrt{var(\hat{p})}, \quad \hat{q} \pm Z_{\frac{\lambda}{2}}\sqrt{var(\hat{q})},
$$

\noindent
where  $\lambda >0$, and $Z_{\frac{\lambda }{2}}$ denote the upper $(\frac{\lambda }{2})$-th percentile of standard normal distribution.


\section{Application}
In this Section, we will analysis  a real data set using the BFWD $(\alpha, \beta, p, q)$ and compare it with the other fitted distributions such as flexible Weibull extension distribution (FWED), Weibull distribution (WD), Modified Weibull distribution (MWD), Reduced Additive Weibull distribution (RAWD) and Extended Weibull distribution (EWD) by using some criteria statistical such as  K--S statistics ( Kolmogorov Smirnov),  Akaike information criterion $AIC$ and Akaike Information criterion with correction $AICC$, see \cite{Akaike1974}. Moreover, we use Bayesian information criterion $BIC$, see  \cite{Schwarz1978}.

We will use the data given by Salman et  al.  \cite{Salmanetal1999}, in Table 1.

\begin{center}
Table 1: Time between failures of secondary reactor pumps, (thousands of hours) \cite{Salmanetal1999}\\
 \begin{tabular}{c c c c c c c c } 
\hline\hline
2.160	& 0.746   & 0.402   & 0.954	& 0.491 & 6.560	& 4.992  & 0.347 \\
0.150   & 0.358   & 0.101	& 1.359	& 3.465	& 1.060	& 0.614  & 1.921 \\	
4.082   & 0.199	  & 0.605	& 0.273 & 0.070	& 0.062	& 5.320  &  \\	\hline
\hline
\end{tabular}
\label{1}
\end{center}

\noindent
Table 2 gives MLEs of parameters of the BFWD$(\eta )$ and the Statistics K--S. The values of $AIC$, $AICC$, $BIC$, $HQIC$ and the log-likelihood functions $\mathcal{L}$  are in Table 3.

\begin{center}
Table 2: MLEs and K--S of parameters for secondary reactor pumps.\\
\begin{tabular}{lccccc} \hline \hline
Model & $\hat{\alpha}$ & $\hat{\beta}$ & $\hat{p }$ &$\hat{q }$ &  K--S\\ \hline
FWD   & 0.0207	    & 2.5875     &  --      &  --    & 0.1342\\
WD	  & 0.8077      & 13.9148	 &  --	    &  --    & 0.1173\\
MWD   & 0.1213	    & 0.7924	 & 0.0009	&  --    & 0.1188\\
RAWD  & 0.0070	    & 1.7292	 & 0.0452	&  --    &  0.1619\\
EWD	  & 0.4189	    & 1.0212	 & 10.2778	&  --    & 0.1057\\
BFWD  & 0.052     & 0.024 & 35.077&  20.328  & 0.1151 \\ \hline \hline
\end{tabular}
\end{center}

\vspace{0.2 cm}

\begin{center}
Table 3: Log-likelihood $\mathcal{L}$, $AIC$, $AICC$,$ BIC$ and $HQIC$ values of models fitted.\\
\begin{tabular}{lcccccc} \hline \hline
Model & $\mathcal{L}$ &-2 $\mathcal{L}$ &	AIC	& AICC	& BIC  & HQIC \\ \hline
FWD	  & -83.3424  & 166.6848 & 170.6848 & 171.2848 & 172.95579 & 171.2559\\
WD    & -85.4734  & 170.9468 & 174.9468 & 175.5468 & 177.21779 & 175.5179\\
MWD	  & -85.4677  & 170.9354 & 176.9354 & 178.1986 & 180.34188 & 177.7921\\
RAWD  & -86.0728  & 172.1456 & 178.1456 & 179.4088 & 181.55208 & 179.0023\\
EWD	  & -86.6343  & 173.2686 & 179.2686 & 180.5318 & 182.67508 & 180.1253\\
BFWD  & -30.768	  &	61.5360	&	69.5360	&	71.7582	&	74.0780	&	70.6783	\\ \hline \hline
\end{tabular}
\end{center}

We find that the BFWD$(\eta )$  with the  parameters $( \alpha, \beta, p, q) $ gives a better fit to data set than the previous generalizations of Weibull distributions such as a flexible Weibull distribution FWD, Weibull distribution WD, Modified Weibull distribution MWD, Reduced Additive Weibull distribution RAWD and Extended Weibull distribution EWD. It has the largest likelihood function $L$, and the smallest  criteria values $AIC$, $AICC$, $BIC$, $HQIC$ and K--S  as shown in Tables 2 and 3.\\

Replace the MLEs of the unknown parameters for BFWD$ (\alpha, \beta, p , q) $  into Eq. (\ref{39}), we can obtained estimation of the variance covariance matrix as follows

$$
I_0^{-1}=\left(
\begin{array}{cccc}
2.123 \times 10^{-3}	 & 9.575 \times 10^{-4}	& -2.748	& -1.6 \\
9.575 \times 10^{-4} 	& 5.558 \times 10^{-4}	& -1.415	& -0.81\\
-2.748	& -1.415	& 3.912 \times 10^{3}	& 2.256 \times 10^{3}\\
-1.6	& -0.81	& 2.256 \times 10^{3}	& 1.304 \times 10^{3}\\
\end{array}
\right).
$$

The confidence intervals for  approximately 95\% two sided  of the unknown parameters $\alpha, \beta, p$ and $q$ are $\left[0, 0.142\right]$, $\left[0, 0.07\right]$, $\left[0, 157.671\right]$ and  $\left[0, 91.105\right]$, respectively.\\

\noindent
From,  Figures \ref{fig:4} and \ref{fig:5} can be seen that the log-likelihood function $\mathcal{L}$ have unique solution.

\begin{figure}[!h] 
\begin{center}
		\includegraphics[width=6cm,height=5cm]{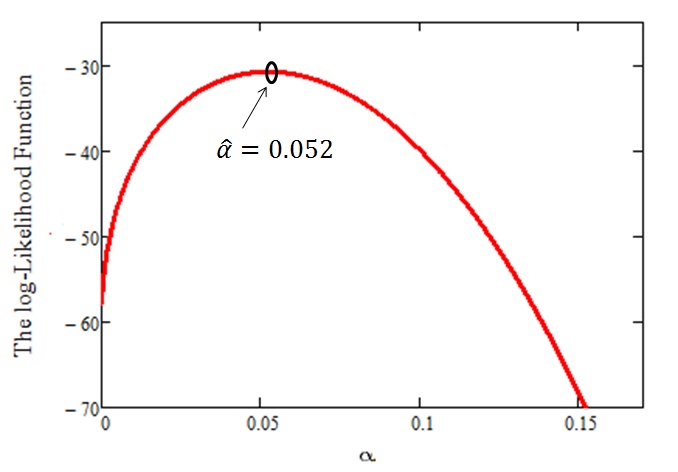}
		\includegraphics[width=6cm,height=5cm]{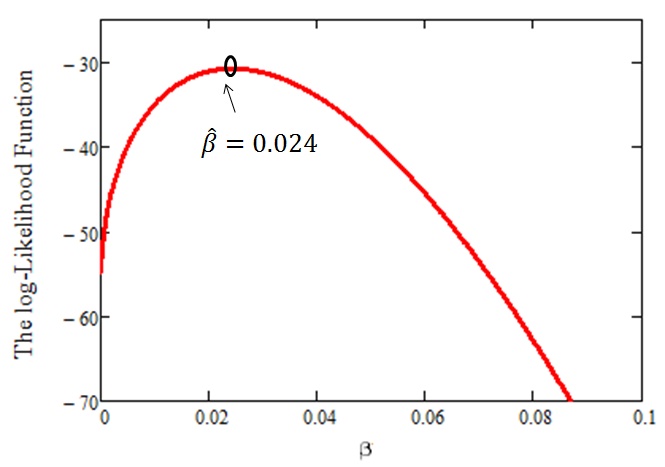}
\end{center}
\caption{The profile of the log-likelihood function of $\alpha, \beta$.} \label{fig:4}
\end{figure}

\begin{figure} [!h]
\begin{center}
		\includegraphics[width=6cm,height=5cm]{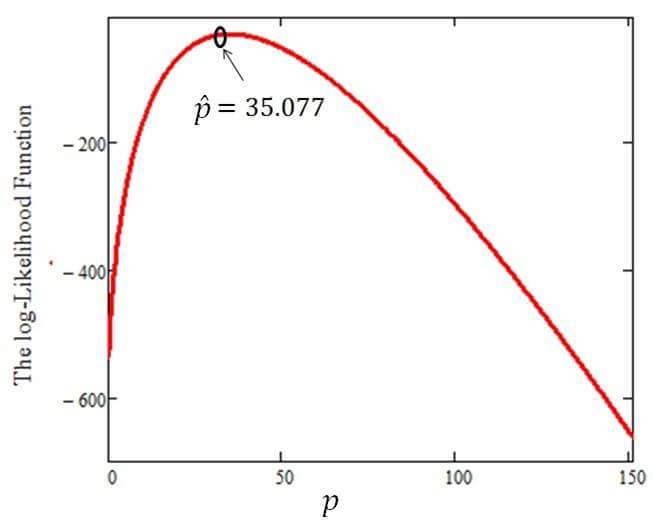}
		\includegraphics[width=6cm,height=5cm]{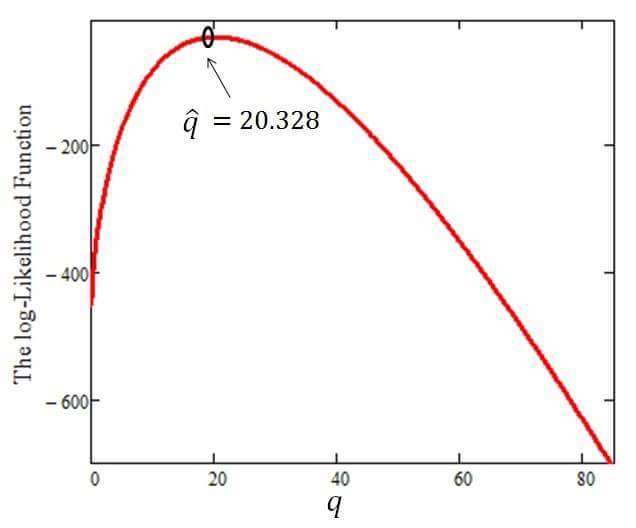}
\end{center}
\caption{The profile of the log-likelihood function of $p$ and $q$.}\label{fig:5}
\end{figure}

\noindent
Figure \ref{fig:6}, represents the estimation for the reliability function $S(x)$, by using the Kaplan-Meier method and its fitted parametric estimations when the distribution is assumed to be  FWD, WD, MWD, RAWD, EWD and BFWD are computed and plotted in the following shape.

\begin{figure} [!h]
\begin{center}
\includegraphics[scale=0.52]{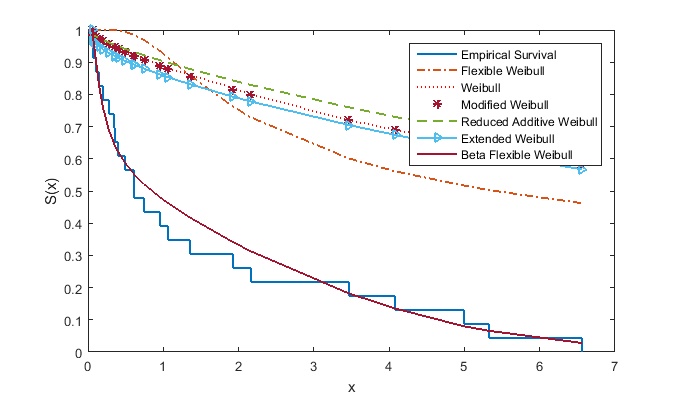}
\end{center}
\caption{Display the Kaplan-Meier estimate of the reliability function for the data.}
\label{fig:6}
\end{figure}

\noindent
Figure \ref{fig:7} gives the shape of the distribution function for the  FWD, WD, MWD, RAWD, EWD and BFWD after that  the unknown parameters included in each distribution are replaced by their maximum likelihood estimation MLE

\begin{figure} [!h]
\begin{center}
\includegraphics[scale=0.52]{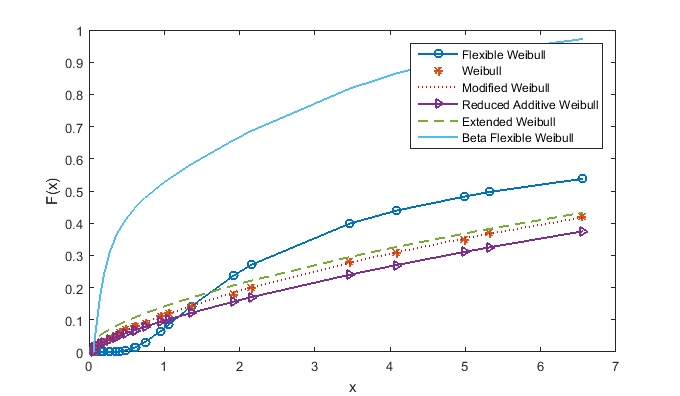}
\end{center}
\caption{The Fitted cumulative distribution function for the data.}
 \label{fig:7}
\end{figure}

\newpage

\section{Conclusion}
We studied a new generalized distribution, based on the beta generated method. This new generalization is called the beta flexible Weibull BFWD distribution . Its definition and  some of statistical properties are studied.  The maximum likelihood method is used for estimating parameters.  The advantage of the BFWD is interpreted by an application using real data. Moreover, it is  shown that the beta flexible Weibull BFWD distribution  fits better than existing generalizations of the Weibull distribution.


\end{document}